\documentclass[12pt]{amsart}

\usepackage{amsmath,amssymb,amsbsy,amsfonts,amsthm,latexsym,
                        amsopn,amstext,amsxtra,euscript,amscd,mathrsfs,color,bm,cite}
                       
\usepackage{float} 
\usepackage[english]{babel}
\usepackage{mathtools}
\usepackage{todonotes}
\usepackage{url}
\usepackage[colorlinks,linkcolor=blue,anchorcolor=blue,citecolor=blue,backref=page]{hyperref}

\usepackage[hmargin=3cm,vmargin=3cm]{geometry}

\usepackage{mathrsfs}

\usepackage{enumitem}

\def\le{\leqslant}
\def\ge{\geqslant}

\def\leq{\leqslant}

\usepackage{mathtools}
\usepackage{todonotes}
\usepackage[norefs,nocites]{refcheck}

\newtheorem{theorem}{Theorem}
\newtheorem{lemma}[theorem]{Lemma}

\theoremstyle{remark}
\newtheorem{rem}[theorem]{Remark}

\newtheorem{definition}[theorem]{Definition}
\newtheorem{example}[theorem]{Example}

\numberwithin{equation}{section}
\numberwithin{theorem}{section}
\numberwithin{table}{section}

\numberwithin{figure}{section}

\def\wt {\mathrm{wt}}
\def\Tr {\mathrm{Tr}}

\def\Fq{\F_q}
\def\Fqr{\F_{q^r}} 
\def\ovFq{\overline{\F_q}}

\def \Grs{\cG_r(s)}

\def\\{\cr}
\def\({\left(}
\def\){\right)}
\def\[{\left[}
\def\]{\right]}
\def\<{\langle}
\def\>{\rangle}
\def\fl#1{\left\lfloor#1\right\rfloor}
\def\rf#1{\left\lceil#1\right\rceil}

\def\cF{{\mathcal F}}
\def\cG{{\mathcal G}}
\def\cH{{\mathcal H}}

\def\cJ{{\mathcal J}}

\def\cL{{\mathcal L}}

\def\cS{{\mathcal S}}

\def\cV{{\mathcal V}}
\def\cW{{\mathcal W}}

\def\e{{\mathbf{\,e}}}
\def\ep{{\mathbf{\,e}}_p}

\def\F{\mathbb{F}}

\def\Z{\mathbb{Z}}

\def\mand{\qquad \text{and} \qquad}

\def\eqref#1{(\ref{#1})}

\begin{document}

\title[Character sums over sparse elements] {Character sums
 over sparse elements of finite fields} 

\author[L. M\'erai]{L\'aszl\'o M\'erai}
\address{L.M.: Johann Radon Institute for Computational and Applied Mathematics, Austrian Academy of Sciences, 4040 Linz, Austria}
\email{laszlo.merai@oeaw.ac.at}


\author[I. E. Shparlinski] {Igor E. Shparlinski}
\address{I.E.S.: School of Mathematics and Statistics, University of New South Wales, Sydney NSW 2052, Australia}
\email{igor.shparlinski@unsw.edu.au}


\author[A. Winterhof]{Arne Winterhof}
\address{A.W.: Johann Radon Institute for Computational and Applied Mathematics, Austrian Academy of Sciences, 4040 Linz, Austria}
\email{arne.winterhof@oeaw.ac.at}

\date{\today}

\begin{abstract}
We estimate mixed character sums of polynomial values over elements of a finite field $\F_{q^r}$ with sparse representations in a fixed ordered basis
over the subfield~$\F_q$. First we use a combination of the inclusion-exclusion principle with bounds on character sums over linear subspaces to get nontrivial bounds for large $q$. Then we focus on 
the particular case $q=2$, which is more intricate.
The bounds depend on certain natural restrictions. We also provide families of  examples for which the conditions of our bounds are fulfilled. In particular, we completely classify all monomials as argument of the additive character for which our bound
is applicable. 
Moreover, we also show that it is applicable for a large family of rational functions, which includes all reciprocal monomials. 

\end{abstract}


\maketitle


\section{Introduction}
\subsection{Motivation and set-up} 
Recently, motivated by a solution to the Gelfond problem by Mauduit and Rivat~\cite{MaRi2}, 
 there has been an explosion in the investigation of arithmetic problems involving 
 integers with various digits restrictions in a given integer base $g$. For example, 
 
 \begin{itemize}
\item Bourgain~\cite{Bou1, Bou2} has investigated primes with prescribed digits
on a positive proportion of positions in their digital expansion, while Maynard~\cite{May1,May2} has investigated primes with missing digits, see also~\cite{Bug,DES,DMR,Nas, Pratt} for a series of other results 
about primes and other arithmetically interesting integers such as smooth and squarefree numbers with restrictions on their digits;

\item Mauduit and Rivat~\cite{MaRi1} and Maynard~\cite{May2} have also studied values of integral polynomials with digital restrictions;

\item Bounds of exponential sums with digitally restricted integers can be found in~\cite{Big,Emi3,OstShp,MaRa,PfTh,ShTh,ThTi}; in some of these, also  the Goldbach and Waring problems with such numbers has been considered;

\item {\'E}minyan~\cite{Emi2} studied average values of arithmetic functions for such numbers. 
\end{itemize}

We also note that special integers with restricted digits appear in the context of cryptography~\cite{GrShp}.

In the case of function fields, several significant results have also been 
obtained, see~\cite{DMS,DMW,DS13,DES,Gab,MW,Ost,Swa1,Swa2} and references 
therein. However, in general, this direction falls behind its counterpart over $\Z$. 
Here we make a step towards removing this disparity and, in particular, we obtain finite 
field analogues  of the bound of~\cite{ShTh} on {\it Weyl sums} over integers with 
``sparse'' binary representations, that is, with a small  sum of binary digits.
More precisely, for an integer 
$$n=\sum_{j=0}^\infty n_j2^j\quad\mbox{with }n_j\in \{0,1\},$$
let
$$\sigma(n)=\sum_{j=0}^\infty n_j$$
be the sum of binary digits of $n$. For any integers $r$ and $s$ with $0\le s\le r$ let
$$\cH_r(s)=\left\{0\le n<2^r :~\sigma(n)=s\right\}$$
be the set of integers with $r$ binary digits and sum of digits equal to $s$. 
For example, 
 in~\cite{ShTh}, 
 one can  find new nontrivial bounds on Weyl sums over $\cH_r(s)$ with  a real polynomial $f(X)$, of the type 
$$\left|\sum_{n\in \cH_r(s)}\exp(2\pi i f(n)) \right| \le C \binom{r}{s}^{-\eta}$$
with $\eta >0$ and $C$ which depends on the ``sparsity'' $s/r$, $\deg f$ and the Diophantine 
properties of the leading coefficient of $f(X)$.

Motivated by these results, we now consider similar questions in the setting of finite fields. Namely, 
we assume that we are 
given a fixed ordered basis $\(\vartheta_1, \ldots, \vartheta_r\)$ of the finite field
$$
\Fqr =\left \{u_1\vartheta_1+ \ldots +u_r \vartheta_r:~u_1, \ldots, u_r \in \Fq\right\}
$$
of $q^r$ elements over the finite field $\Fq$ of $q$ elements. 
Furthermore, for 
$$
\nu= u_1\vartheta_1+ \ldots +u_r \vartheta_r\in \Fqr
$$ 
we denote by $\wt(\nu)$ the {\it Hamming weight\/} of  $\nu$, that is, the number of nonzero elements 
among the components $u_1, \ldots , u_r \in \Fq$. 
Then, for a positive integer $s \le r$ we define the set of $s$-sparse elements
$$
\Grs = \{\nu \in \Fqr:~\wt(\nu) =s\}.
$$
Our goal is to estimate mixed character sums 
$$
S_{s,r}(\chi,\psi;f_1,f_2) = \sum_{\nu \in  \Grs}\chi\(f_1(\nu)\)\psi\(f_2(\nu)\), 
$$
with rational functions $f_1(X), f_2(X) \in \Fqr(X)$, of degrees $d_1$ and $d_2$, respectively, where~$\chi$ and~$\psi$ are a fixed multiplicative and 
 additive character of $\F_{q^r}$, respectively (with the natural conventions that the poles of $f_1(X)f_2(X)$ are excluded from summation,
 and that a rational function $f(X)=a(X)/b(X)$ is represented by relatively prime polynomials $a(X)$ and $b(X)$). 
 
 Certainly our bounds can be used  to study, for example, 
 the distribution of primitive elements in the values of polynomials on elements from 
 $\Grs$ or their pseudorandom properties. Since these underlying methods are very standard we do not give these applications here.


\subsection{Notation and conventions}

Throughout the paper,  we fix the size $q$ of the ground field, and thus its characteristic  $p$ 
while the parameters $r$ and $s$ are allowed to grow. 

We also fix an additive character $\psi$ and a multiplicative character~$\chi$ of~$\Fqr$
which are not both principal.

As usual, we use $\ovFq$ to denote the algebraic closure of $\F_q$. 

For a finite set $\cS$ we use $\# \cS$ to denote its cardinality.

We denote by $\log_2 x$ the binary logarithm of $x>0$.

We adopt the Vinogradov symbol $\ll$,  that is, for any quantities $A$ and $B$ we
have the following equivalent definitions:
$$A\ll B~\Longleftrightarrow~A=O(B)~\Longleftrightarrow~|A|\le c B$$
for some  constant $c>0$, which throughout the paper is allowed to depend on the 
degrees $d_1,d_2$ and the ground field size~$q$ 
(but not on the main parameters $r$ and $s$).

 We also adopt the $o$-notation
$$A=o(B)~\Longleftrightarrow ~|A|\le \varepsilon B$$
for any fixed $\varepsilon>0$ and sufficiently large (depending on $d_1$, $d_2$, $q$ 
and~$\varepsilon$) values of the parameters $r$ and $s$.

We also write 
$$\ep(z)=\exp(2\pi iz/p).
$$
Finally, we also recall our convention that the poles of functions in the arguments of multiplicative and 
additive characters are always excluded from summation.

\section{Our results}

\subsection{Bounds of exponential sums over sparse elements} 


It is useful to recall that the (absolute) {\em trace} $\Tr(\xi)$ of $\xi\in \F_{q^r}$, is
$$\Tr(\xi)=\sum_{j=0}^{rm-1}\xi^{p^j},$$ 
where $q=p^m$ with a prime $p$, and 
the  {\em additive characters} $\psi$ of $\F_{q^r}$ are given by
$$\psi (\xi)=\ep(\Tr(\zeta \xi)),\qquad \xi\in \F_{q^r},$$
for some fixed $\zeta \in \F_{q^r}$ (with obviously $\zeta \ne 0$ for nonprincipal characters).
The multiplicative characters of $\F_{q^r}$ are 
given by
$$\chi(\gamma^j)=\e_{q^r-1}(jk),\quad j,k=0,1,\ldots,q-2,$$
for a primitive element $\gamma$ of $\F_{q^r}$ (with $k\not=0$ for nonprincipal characters). We also use the convention $\chi(0)=0$.
We also recall that the {\em order} of a multiplicative character $\chi$ is the smallest integer 
exponent $e\ge 1$ such that $\chi^e = \chi_0$ is the  principal character.
See, for example,~\cite[Chapters~3 and~11]{IwKow} or~\cite[Chapter~5]{LiNi} 
for more details on characters and character sums over finite fields.

Trivially we have
$$
\left|S_{s,r}(\chi,\psi;f_1,f_2)\right|\le \#\cG_r(s)=\binom{r}{s}(q-1)^s.
$$

First we record a rather simple result which we derive in Section~\ref{sec:largeq}, 
combining the inclusion-exclusion principle and bounds on character sums over linear subspaces.  

\begin{theorem}\label{thm:largeq}
Let $\chi$ and $\psi$ be a multiplicative and additive character, respectively, and 
let $f_1(X) , f_2(X)\in \F_{q^r}(X)$ be rational functions over~$\F_{q^r}$ of degrees $d_1$ and $d_2$, respectively. 
Assume that at least one of the following conditions holds
\begin{itemize}
\item $\chi$ is nonprincipal of order $e$ and $f_1(X) \ne g(X)^e$ for all rational functions $g(X)\in\ovFq(X)$, 
\item $\psi$ is nonprincipal and $f_2(X)  \ne \alpha(g(X)^p-g(X))+\beta X$ for all rational functions $g(X)\in\ovFq(X)$
and $\alpha, \beta \in \ovFq$. 
\end{itemize}
Then we have
$$\left|S_{s,r}(\chi,\psi;f_1,f_2)\right|\le \(d_1+\max\{d_2,2\}\) 2^{s+1}\binom{r}{s}q^{r/2}.$$
\end{theorem}

\begin{rem}
We note that the conditions on $f_1(X)$ and $f_2(X)$ in  Theorem~\ref{thm:largeq} are natural 
and correspond to the conditions of Lemma~\ref{lem:weilsub} below 
(which in turn stems from a similar necessary and sufficient condition of 
 Lemma~\ref{lem:weil}).  Note that Theorem~\ref{thm:largeq} is applicable if 
  $\chi$ is nonprincipal of order $e\nmid d_1$ which is always true if $\gcd(d_1,q^r-1)=1$.
It is also always applicable  if $f_1(X)$ has at least one simple root in~$\ovFq$.
 Furthermore,  if $\psi$ is nonprincipal, Theorem~\ref{thm:largeq}  always 
 applies when~$d_2\ge 2$  and $\gcd(d_2,p)=1$.
\end{rem}

\begin{rem} If $\psi= \psi_0$ is the trivial character, that is in the case of pure sums of multiplicative characters 
$$
S_{s,r}(\chi;f) = \sum_{\nu \in  \Grs}\chi\(f(\nu)\),
$$
using a slightly more precise version of  Lemma~\ref{lem:weilsub} for  pure sums of multiplicative 
characters over linear subspaces, see for example,~\cite[p.~469, Claim~(iii)]{Win},
one can replace the bound 
of  Theorem~\ref{thm:largeq}  with 
$$|S_{s,r}(\chi;f)|\le 2^s t\binom{r}{s}q^{r/2}, $$
where $t$ is the number of  distinct zeros and poles of $f(X) \in \Fq(X)$ in $\ovFq$ (provided, as before, that 
$f(X) \ne g(X)^e$ for all rational functions $g(X)\in\ovFq(X)$,  where $e$ is the order of $\chi$).  
\end{rem}

 Assuming that $d_1$, $d_2$ and $q$ are fixed, we see that the bounds of Theorem~\ref{thm:largeq} 
 are nontrivial if $r,s \to \infty$ and  satisfy
 $$
 \frac{s}{r} > \frac{\log_2 q }{2\log_2((q-1)/2)},\quad q>3.
 $$
 It is useful to observe that 
 $$
 \frac{\log_2 q}{2\log_2((q-1)/2)} < 1, \quad q\ge 7,
 $$
 and thus Theorem~\ref{thm:largeq}  produces nontrivial results, 
 starting from $q \ge 7$.

Next we  focus on our main goal when $q=2$ and the case that the degree $r$ of the extension field $\F_{q^r}$ grows. 

We now define the following class $\cF_d(q)$ of rational functions
for which we obtain nontrivial estimates of the sums $S_{s,r}(\chi,\psi;f_1,f_2)$ in case of $\psi\not=\psi_0$.

\begin{definition}  
\label{def'Fdq}
 Let $\cF_d(q)$ be the set of rational functions $f(X)\in\F_{q^r}(X)$ of 
degree $d$ such that 
for any  $\omega\in \F_{q^r}^*$
the polynomial 
$$
f_{\omega}(X)=f(X+\omega)-f(X)
$$
is not of the form
$$
f_{\omega}(X)=\alpha\(g(X)^p-g(X)\)+\beta X 
$$
for some rational function $g(X)\in \ovFq(X)$  and  $\alpha,\beta\in \F_{q^r}$.
\end{definition}    

We repeatedly use the observation that for $f(X) \in \cF_d(q)$ we also have a more general condition that for any $\omega_1\not=\omega_2$
$$f(X+\omega_1)-f(X+\omega_2) \ne \alpha\(g(X)^p-g(X)\)+\beta X,
$$
where $\alpha$, $\beta$ and $g(X)$ are as in Definition~\ref{def'Fdq}.

Next we concentrate on the case $q=2$. 
We remark that our approach works for any  fixed $q$, however 
our ideas can be explained in a much more transparent form for rational functions over $\F_2$,
hence we only consider this case. We also note that classifying elements by their 
Hamming weight is the most natural for $q=2$.

Let $H(\gamma)$ be the {\it binary entropy\/} function defined by
\begin{equation}
\label{eq:bin entrop}
    H(\gamma)=-\gamma\log_2 \gamma -(
    1-\gamma)\log_2(1-\gamma),  \qquad 0<\gamma<1,
\end{equation}
and 
$H(0)=H(1)=0$.
It is also convenient to define 
$$
H^*(\gamma) =\begin{cases}  H(\gamma),& \text{if}\  0 \le \gamma\le \frac{1}{2};\\
H\(\frac{1}{2}\)=1,  & \text{if}\    \gamma>\frac{1}{2}.
\end{cases} 
$$
We now define the following three functions depending 
on three   parameters $\kappa,\lambda, \rho \in (0,1)$ with $\lambda \le \kappa \rho \le \frac{\kappa}{2}$:
\begin{align*}
\mathsf f(\rho; \kappa, \lambda) 
 &=\kappa H\(\lambda/\kappa\) +  (1-\kappa) H^*\((\rho -\lambda)/(1-\kappa)\),\\
\mathsf g(\rho; \kappa, \lambda)  &=  \frac{1}{4} + \frac{1}{2}  \(H(\rho)+ (1-\kappa)H^*\((\rho-\lambda)/(1-\kappa)\)\),\\
\mathsf h(\rho; \kappa, \lambda)  &= \frac{1}{2} \( H(\rho)+ \kappa\).
\end{align*}
Finally, we define 
$$
\eta(\rho) = \min_{0< \kappa <1}\,  \min_{0 < \lambda \le \kappa \rho}
\max\{\mathsf f(\rho; \kappa, \lambda) , \mathsf g(\rho; \kappa, \lambda) , \mathsf h(\rho; \kappa, \lambda) \}.
$$

\begin{theorem}\label{thm:general}
Let $\chi$ and $\psi$ be a multiplicative and additive character, respectively, and 
let $f_1(X) , f_2(X)\in \F_{2^r}(X)$ be rational functions over~$\F_{2^r}$ of degrees $d_1$ and $d_2$, respectively. 
Assume that at least one of the following conditions holds
\begin{itemize}
\item $\chi$ is nonprincipal and $f_1(X)$ has least one simple root in $\overline{\F_2}$,  
\item $\psi$ is nonprincipal and $f_2(X) \in\cF_{d_2}(2)$. 
\end{itemize}
For $s$ with $1\le s < r$ we put
$$\rho=\frac{\min\{s,r-s\}}{r}.$$
Then 
$$
\left|S_{s,r}(\chi,\psi;f_1,f_2)\right|\le  2^{\eta(\rho) r + o(r)}.
$$ 
\end{theorem} 

It is useful to observe that we have $\rho \le 1/2$ in Theorem~\ref{thm:general}.

Since  under the conditions of Theorem~\ref{thm:general}, using the standard bound on the growth of binomial coefficients,  we have
$$
\#\cG_r(s) = 2^{H(\rho) r + o(r)}
$$
see, for example,~\cite[Chapter~10, Lemma~7]{Sloane}, 
and we conclude that Theorem~\ref{thm:general} is nontrivial if $\eta(\rho) < H(\rho)$.

The following plot in Figure~\ref{fig:OldNew} compares the bounds of Theorem~\ref{thm:general} with the trivial bound $2^{H(\rho)}$, more precisely we compare $\eta(\rho)$ with $H(\rho)$. In particular, these bounds are nontrivial for 
$\rho \ge 1/5$.
More precisely, we have $H\(1/5\)> 0.7219$ and $\eta\(1/5\)<0.7208$.
Thus for $\rho =  1/5$ we save more than $0.001$ against the trivial bound.
(Finding the explicit intersection point $\rho$ with $H(\rho)=\eta(\rho)$ is computationally very extensive.)

\begin{figure}[h]
\centering
\includegraphics[scale=0.78]{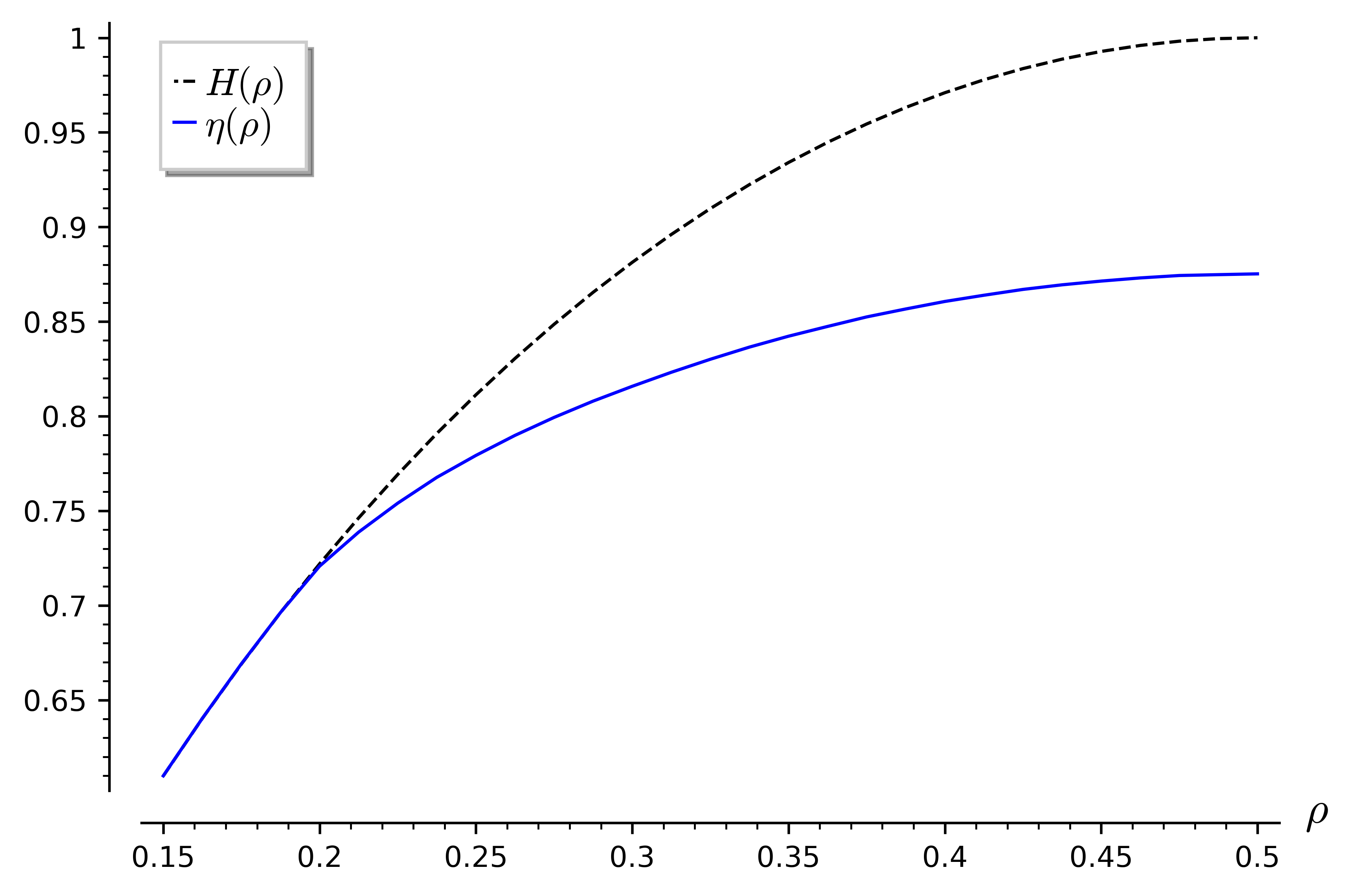}
 \caption{
$H(\rho)$ vs.\ $\eta(\rho)$
}
 \label{fig:OldNew}
\end{figure}

The function $\eta(\rho)$ is not easy to describe in a concise form. Hence, we give 
a simplified version of  Theorem~\ref{thm:general}, which is still nontrivial in a wide range of $\rho$.
Choosing $\kappa=3/4$ and $\lambda$ small enough, more precisely, any $\lambda>0$ with 
$$
H\(\frac{4\lambda}{3}\)\le \frac{4H(\rho)+1}{6},
$$ 
we get the bound 
\begin{equation}\label{etaH}
\eta(\rho)\le \frac{H(\rho)}{2}+\frac{3}{8},
\end{equation}
which is nontrivial if  $\rho> \rho_0$ where  $\rho_0 = 0.2145\ldots$ is 
the root of the equation $H(\rho_0) =  3/4$.
Note that $\eta(\rho)$ is strictly smaller than the right hand side of~\eqref{etaH} for~$\rho\le 0.275$.

It is easy to generalize this bound to any (fixed) $q>2$:
$$S_{s,r}(f)\ll \#\Grs^{1/2}q^{3r/8+o(r)}\le (2(q-1))^{H(\rho)r/2}q^{3r/8+o(r)}.$$

\begin{rem} Instead of using the Cauchy-Schwarz inequality in the proof of Theorem~\ref{thm:ratsuff}, more generally, one 
 can  also use the H\"older inequality. This leads us to investigate the frequency of the vectors 
$(\omega_1, \ldots, \omega_{2k})\in \F_{q^r}^{2k}$ for which 
$$
\sum_{j=1}^k \(f(X+\omega_j)-f(X+\omega_{k+j}) \) = \alpha\(g(X)^p-g(X)\)+\beta X 
$$
for some rational function $g(X)\in \ovFq(X)$  and  $\alpha,\beta\in \F_{q^r}$. Preferably, we  would like this to happen only if $\{\omega_1, \ldots, \omega_k\} = \{ \omega_{k+1}, \ldots, \omega_{2k})$. Although this may potentially lead to an extension of the range of parameters for which we have nontrivial bounds on the sums $S_{s,r}(\chi,\psi;f_1,f_2)$, since we do not have natural examples of polynomials satisfying  the above condition, we do not pursue this direction here.  
\end{rem}

We prove Theorem~\ref{thm:general} in Section~\ref{sec:general}.

\subsection{Families of rational functions from $\cF_d(q)$} 
Finally we discuss several cases of polynomials for which the conditions on $f_2(X)$ of Theorem~\ref{thm:general} are satisfied.
We start with monomials $f(X)=X^d$. 
Since 
$$\Tr(\alpha f(\xi))=\Tr(\alpha^p f(\xi)^p), \qquad \xi\in \F_{q^r},$$
the maxima over all nonprincipal characters  $\chi$ and $\psi$ of  the sums
$|S_{s,r}(\chi,\psi;f_1,f_2) |$ and $|S_{s,r}(\chi,\psi;f_1,f_2^p) |$ 
are the same. 
Therefore, we may restrict ourselves to the case $\gcd(d,p)=1$.

\begin{theorem}\label{thm:mon}
For any positive integer $d$ with $\gcd(d,p)=1$ 
and monomials $f(X)=X^{d}$, we have $f(X) \in \cF_d(q)$ if and only if 
$$
d\not\in \{1,2\}\cup \{p^k+1:~k=1,2,\ldots\}.
$$
\end{theorem} 
We prove Theorem~\ref{thm:mon} in Section~\ref{sec:mon}.

For reciprocal monomials we have the following result. Again we may restrict ourselves to the case $\gcd(d,p)=1$.
\begin{theorem}\label{thm:inv}
For any positive integer $d$ with $\gcd(d,p)=1$ and
$$f(X)=X^{-d}$$ 
we have
$$f(X)\in  \cF_d(q).$$  
\end{theorem}
We prove Theorem~\ref{thm:inv} in Section~\ref{sec:inv}.

For arbitrary polynomials we get the following general result (which 
unfortunately is void for $p=2$). 

\begin{theorem}\label{thm:pol} 
Let $p$ be the characteristic of $\F_{q^r}$ and let 
$d \equiv d_0 \bmod p$ with some integer $d_0 \in [2, p-1]$.
Then for any polynomial $f(X)\in \F_{q^r}[X]$ of degree 
$d\ge 3$ we have $f(X) \in \cF_d(q)$. 
\end{theorem}

We prove Theorem~\ref{thm:pol} in Section~\ref{sec:pol}.

Finally, we prove the following sufficient condition for $f(X)\in \cF_d(q)$ for a rational function $f(X)$ in Section~\ref{sec:suff}. In contrast to Theorem~\ref{thm:pol}, it applies also to $p=2$.
\begin{theorem}\label{thm:ratsuff}
  Let $p$ be the characteristic of $\F_q$ and
  let $f(X)=u(X)/v(X)$ be a rational function with nonzero polynomials $u(X),v(X)\in \F_{q^r}[X]$ with $\gcd(u(X),v(X))=1$,  and such that
  $$\deg u\le \deg v+1 \mand \gcd\(\deg v,p\)=1.$$
 Assume that for all $\omega\in \F_{q^r}^*$, 
$ v(X)v(X+\omega)$ is not divisible by the $p$th power of a nonconstant polynomial.

  Then $f(X)\in \cF_d(q)$.
\end{theorem}
\begin{rem} The non-divisibility condition of Theorem~\ref{thm:ratsuff} is fulfilled if either $\deg v<\frac{p}{2}$
(which is again void for $p=2$) or $v(X)$ is irreducible and $\gcd(\deg v,p)=1$.
For the latter condition verify that the coefficients of $X^{\deg v-1}$ of $v(X)$ and $v(X+\omega)$ are different and thus $v(X)v(X+\omega)$ is squarefree.
\end{rem}

\section{Preliminary results}

\subsection{Bounds of character sums}

Our main tool is the Weil bound for mixed character sums of rational functions, 
which we present in a simplified  form which can be easily extracted 
from~\cite{CaMo} applied to an affine line (which is of genus zero).
See also~\cite[Theorem~5.6]{FW} for much more general result.


\begin{lemma}\label{lem:weil}  Let $\chi$ and $\psi$ be a multiplicative and additive character, respectively, and 
let $g_1(X) , g_2(X)\in \F_{q^r}(X)$ be rational functions of degrees $D_1$ and $D_2$ over $\F_{q^r}$, respectively. 
Assume that at least one of the following conditions holds
\begin{itemize}
\item $\chi$ is nonprincipal of order $e$ and $g_1(X)  \ne h(X)^e$ for all rational functions $h(X)\in\ovFq(X)$,
\item $\psi$ is nonprincipal and $g_2(X)  \ne \alpha\(h(X)^p-h(X)\)$ for all rational functions $h(X)\in\ovFq(X)$
and $\alpha \in \ovFq$. 
\end{itemize}
 Then we have
$$\left|\sum_{\xi \in \F_{q^r}}\chi(g_1(\xi))\psi\(g_2(\xi)\)\right| \le 
2(D_1+D_2)q^{r/2}.$$
\end{lemma}

We need the following version of~\cite[p.~470, Claim~(iii)]{Win}.

\begin{lemma}\label{lem:weilsub} Let $\chi$ and $\psi$ be a multiplicative and additive character, respectively, and 
let $g_1(X) , g_2(X)\in \F_{q^r}(X)$ be rational functions of degrees $D_1$ and $D_2$ over $\F_{q^r}$, respectively. 
Assume that at least one of the following conditions holds
\begin{itemize}
\item $\chi$ is nonprincipal of order $e$ and $g_1(X)  \ne h(X)^e$ for all rational functions $h(X)\in\ovFq(X)$,
\item $\psi$ is nonprincipal and $g_2(X)  \ne \alpha(h(X)^p-h(X))+\beta X$ for all rational functions $h(X)\in\ovFq(X)$
and $\alpha, \beta \in \ovFq$. 
\end{itemize}
Then for any linear subspace of $\Fqr$ 
\begin{equation}\label{Lk} \cL_k = \{u_1\vartheta_1+ \ldots +u_k \vartheta_k:~u_1, \ldots, u_k \in \Fq\}
\end{equation}
of dimension $k \le r$, 
we have 
$$
\left| \sum_{\nu \in  \cL_k}\chi\(g_1(\nu)\)\psi\(g_2(\nu)\)\right|\le 2(D_1+\max\{D_2,2\})q^{r/2}.
$$
\end{lemma}

\begin{proof}
If $\psi$ is a  nonprincipal character, then combining the analogue of~\cite[Lemma~3.1]{Win} for mixed sums and~\cite[Lemma~3.4]{Win} we get
$$\left| \sum_{\nu \in  \cL_k}\chi\(g_1(\nu)\)\psi\(g_2(\nu)\)\right|
\le \max_{\beta\in \F_{q^r}}\left|\sum_{\nu\in \F_{q^r}}\chi\(g_1(\nu)\)\psi\(g_2(\nu)-\beta\nu\)\right|$$
and the result follows from Lemma~\ref{lem:weil}.

Otherwise we have for any nonprincipal additive character $\widetilde \psi$,
$$ \left| \sum_{\nu \in  \cL_k}\chi\(g_1(\nu)\)\psi\(g_2(\nu)\)\right|
\le \max_{\beta\in \F_{q^r}}\left|\sum_{\nu\in \F_{q^r}}\chi\(g_1(\nu)\)\widetilde \psi\(-\beta\nu\)\right| \le  (2D_1+2)q^{r/2}
$$
by Lemma~\ref{lem:weil}.
\end{proof}

It is quite obvious that  the restriction on the polynomial $g_1(X)$ in Lemma~\ref{lem:weilsub} 
is necessary. 
We now  show that the restriction on the polynomial $g_2(X)$ in Lemma~\ref{lem:weilsub} 
is  also necessary.

\begin{example}Assume that $g_2(X)=\alpha(h(X)^p-h(X))+\beta X$ for some $\alpha,\beta\in \F_{q^r}$. We may assume $\alpha\ne 0$.  Consider the character 
$$
\psi(\xi)=\ep\(\Tr\(\alpha^{-1}\xi\)\), \quad \xi\in \F_{q^r}
$$ 
and the trivial character $\chi=\chi_0$. 
Then we have $\Tr(\alpha^{-1}g_2(\xi))=\Tr\(\alpha^{-1}\beta\xi\)$ for $\xi\in \F_{q^r}$. Let 
$\cL_{r-1}$ be the kernel of $\Tr\(\alpha^{-1}\beta X\)$, which has dimension $r-1$. Then for any $\xi\in \cL_{r-1}$ we have $\Tr\(\alpha^{-1}g_2(\xi)\)=0$, thus
$$\sum_{\nu\in \cL_{r-1}}\psi\(g_2(\nu)\)=q^{r-1}.$$
\end{example}

\begin{rem} We note that the results of Ostafe~\cite{Ost} allow us
to estimate character sums over linear subspaces of much lower dimension. 
However, this bound applies only to polynomials and subspaces of a very special form and seems to be not suitable for our purposes.
\end{rem} 

%

\subsection{Some properties of binomial coefficients}

We recall the definition~\eqref{eq:bin entrop} of the entropy function~$H$.
We frequently use the following result from~\cite[Chapter~10, Corollary~9]{Sloane}:

\begin{lemma}
\label{lem:BinCoeffs}
For any natural number $n$ and $0<\gamma \le 1/2$, we have
$$
    \sum_{0\leq k\leq\gamma n}\binom{n}{k}\leq 2^{nH(\gamma)}.
$$
\end{lemma}

We also have the following technical result. 

\begin{lemma}\label{max}
For fixed real $\kappa$ and $\rho$ with
$$0<\kappa<1\quad \mbox{and}\quad 0<\rho\le \frac{1}{2}$$
the function
$$E(\lambda)=\kappa H^*\left(\frac{\lambda}{\kappa}\right)+(1-\kappa)H^*\left(\frac{\rho-\lambda}{1-\kappa}\right), \qquad 0<\lambda <\rho,$$
is monotonically increasing in 
$\left(0,\kappa\rho\right]$
and monotonically decreasing in~
$\left(\kappa \rho,\rho\right)$. 
\end{lemma}

\begin{proof}
It is easy to see that 
$$H'(\gamma)=-\log_2\(\frac{\gamma}{1-\gamma}\),\quad 0<\gamma<1,$$
and thus
$H(\gamma)$ is monotonically increasing in $\(0,\frac{1}{2}\]$.

For 
$$
0<\lambda\le\rho-\frac{1-\kappa}{2} = \frac{\kappa}{2} - \(\frac{1}{2}-\rho\) 
$$
we have
$$E(\lambda)=\kappa H\left(\frac{\lambda}{\kappa}\right)+(1-\kappa)$$
which is monotonically increasing in $\(0,\frac{\kappa}{2}\]$.

For 
$$\max\left\{\rho-\frac{1-\kappa}{2},0\right\}< \lambda\le \frac{\kappa}{2}
$$
we have
$$E(\lambda)=\kappa H\(\frac{\lambda}{\kappa}\)+(1-\kappa)H\(\frac{\rho-\lambda}{1-\kappa}\)$$
and
\begin{align*}
E'(\lambda)&=H'\left(\frac{\lambda}{\kappa}\right)-H'\left(\frac{\rho-\lambda}{1-\kappa}\right)\\
&=- \log_2 \left(\frac{\lambda}{\kappa-\lambda}\right)+
\log_2 \left(\frac{\rho-\lambda}{1-\kappa-\rho+\lambda}\right).
\end{align*}  
It is also useful to note that because $\rho \le 1/2$ we have
$$
\rho-\frac{1-\kappa}{2} \le \kappa \rho.
$$

Verify that 
$$\frac{\rho-\lambda}{1-\kappa-\rho+\lambda}\ge \frac{\lambda}{\kappa-\lambda}
\quad \mbox{if and only if}\quad \lambda\le\rho\kappa$$
and thus
$$E'(\lambda)\ge 0,\qquad \rho-\frac{1-\kappa}{2}\le \lambda\le \rho\kappa,$$
and 
$$E'(\lambda)<0,\qquad \frac{\kappa}{2}\ge \lambda >\rho\kappa.$$ 
Hence, $E(\lambda)$ has a local maximum at $\lambda=\kappa\rho$.

Finally, for 
$$\frac{\kappa}{2}<\lambda< \rho
$$
we get
$$E(\lambda)=\kappa+(1-\kappa)H\left(\frac{\rho-\lambda}{1-\kappa}\right),$$
which is monotonically decreasing since 
$$
0 < \frac{\rho-\lambda}{1-\kappa} < \frac{1}{2}
$$
in this range (recall again that $\rho \le 1/2$).
\end{proof}

Next we state the well-known Lucas congruence, see, for example,~\cite[Lemma~6.3.10]{NiWi}.

\begin{lemma}\label{lem:lucas}
If $m$ and $n$ are two natural numbers with the following $p$-adic expansions
\begin{align*}
&m=m_{r-1}p^{r-1}+\ldots+m_1p+m_0,\qquad 0 \le m_0,\ldots,m_{r-1}< p,\\
&n=n_{r-1}p^{r-1} +\ldots+ n_1p+n_0,\qquad 0 \le n_0,\ldots,n_{r-1}< p,
\end{align*}
then we have
$$
\binom{m}{n} \equiv \prod_{j=0}^{r-1}  \binom{m_j}{n_j} \bmod p.
$$
\end{lemma}

\section{Proofs of bounds on character sums} 
\subsection{Proof of Theorem~\ref{thm:largeq}}\label{sec:largeq}

For $\cS\subseteq\{1,\ldots,r\}$ put 
$$\cG_\cS=\left\{\sum_{j\in \cS} u_j\vartheta_j :~u_j\in \F_q^*,\ j\in \cS\right\}.$$
Since $\cG_r(s)$ is the disjoint union of
different sets $\cG_S$ with $\# \cS=s$ we immediately get
\begin{equation}\label{red1}
\left|S_{s,r}(\chi,\psi;f_1,f_2)\right|\le \binom{r}{s} \max_{\substack{\cS\subseteq\{1,\ldots,r\}\\ \# \cS=s}} \left|S_\cS(\chi,\psi;f_1,f_2)\right|,
\end{equation}
where
$$S_\cS(\chi,\psi;f_1,f_2)=\sum_{\nu \in \cG_{\cS}}\chi\(f_1(\nu)\)\psi\(f_2(\nu)\).$$
Let $\cV_\cS$ denote the linear space
$$\cV_\cS=\left\{\sum_{j\in \cS}u_j\vartheta_j:~u_j\in \F_q, \  j\in \cS\right\}$$
of dimension $s$ and verify 
$$\cG_{\cS}=\cV_{\cS}\setminus \bigcup_{j\in \cS}\cV_{\cS\setminus\{j\}}.$$
By the inclusion-exclusion principle we get
\begin{equation}\label{red2}
\begin{split}
S_\cS(\chi,\psi;f_1,f_2)&=\sum_{\nu\in \cV_\cS} \chi\(f_1(\nu)\)\psi\(f_2(\nu)\)\\
&\qquad -\sum_{\emptyset\not= \cJ\subseteq \cS}(-1)^{|J|+1}\sum_{\nu \in \cV_{\cS\setminus\cJ}}\chi\(f_1(\nu)\)\psi\(f_2(\nu)\).
\end{split}
\end{equation}
Combining~\eqref{red1} and~\eqref{red2}, we obtain
$$
 \left|S_{s,r}(\chi,\psi;f_1,f_2)\right|    \le \binom{r}{s}\left(1+\sum_{k=1}^s \binom{s}{k}\right)\max_\cV \left|\sum_{\nu\in V}\chi\(f_1(\nu)\)\psi\(f_2(\nu)\)\right|,
$$
where the maximum is taken over the $\F_q$-linear subspaces $\cV$ of $\F_{q^r}$.

Since the quadruple $(\chi,\psi;f_1,f_2)$  satisfies the condition of Lemma~\ref{lem:weilsub} we now easily derive the first result.


\subsection{Proof of Theorem~\ref{thm:general}}\label{sec:general}

Note that the required conditions of Theorem~\ref{thm:general} on $f_1(X)$ and $f_2(X)$ are invariant under the shift 
of the argument. For example, 
 $f_2(X)\in \cF_d(2)$ whenever $f_2(X+\vartheta_1+\ldots+\vartheta_r)\in \cF_d(2)$. 
Since
$$
S_{r-s,r}\(\chi,\psi;f_1,f_2\)
 =\sum_{\nu\in G_r(s)}\chi\(f_1\(\nu+\vartheta_1+\ldots+\vartheta_r\)\)\psi\(f_2 \(\nu+\vartheta_1+\ldots+\vartheta_r\)\),
$$
we may restrict ourselves to the case 
$$s\le r/2.$$

For given $\rho$ with $0<\rho \le 1/2$ we fix two real positive  parameters $\kappa$ and $\lambda$ 
and  set
$$
k = \rf{\kappa r}  \mand \ell = \fl{\lambda r} .
$$ 

We now decompose 
$$
S_{s,r}(\chi,\psi;f_1,f_2)   = \sum_{t= 0}^s T_t
$$
into the sums
$$
T_t = \sum_{\nu\in \cV_t} \sum_{\omega \in \cW_t} \chi\(f_1\(\nu+\omega\)\)\psi\(f_2 \(\nu+\omega\)\), 
$$
where 
\begin{align*}
& \cV_t =  \{\nu= u_1\vartheta_1+ \ldots +u_k \vartheta_k:~u_1, \ldots, u_k \in \F_2, \ \wt(\nu) = t\},\\
& \cW_t =  \{\omega= u_{k+1}\vartheta_{k+1}+ \ldots +u_r \vartheta_r:~u_{k+1}, \ldots, u_r \in \F_2, \ \wt(\omega) = s-t\}.
\end{align*}

In particular, we   have
\begin{equation}
\label{eq: Card Vt Wt} 
\#  \cV_t=   \binom{k}{t} \mand  \# \cW_t  =  \binom{r-k}{s-t},\quad t=0,1,\ldots,s.
\end{equation}

Clearly by Lemma~\ref{lem:BinCoeffs} we have 
\begin{equation}
\label{eq: AnyBinCoeff} 
\binom{n}{k}\leq 2^{nH^*(\gamma)}, 
\end{equation}
with  $\gamma =k/n$ for any integers 
$k\ge 0$ and $n\ge 1$.

First, we record the trivial bound 
$$|T_t| \le \#\cV_t \# \cW_t$$  
which we apply for $t <\ell$.

Hence, using~\eqref{eq: Card Vt Wt} and~\eqref{eq: AnyBinCoeff},  we conclude that the contribution 
to~$S_{s,r}(f)$ 
from such values of $t$ is bounded by 
\begin{equation}
\label{eq: Small t} 
 \sum_{t=0}^{\ell-1}|T_t| \le \sum_{t=0}^{\ell-1} \#\cV_t \# \cW_t \le \ell 2^{\mathsf f(\rho; \kappa, \lambda)r}
 =  2^{\mathsf f(\rho; \kappa, \lambda)r + o(r)},
\end{equation}
where
$$\mathsf f(\rho; \kappa, \lambda) 
= \max_{0 < \tau \le \lambda}  \left\{\kappa H^*\(\frac{\tau}{\kappa}\) +  (1-\kappa) H^*\(\frac{\rho -\tau}{1-\kappa}\)\right\}.$$
By Lemma~\ref{max} the maximum is attained for $\tau=\min\{\lambda,\kappa\rho\}$.
In particular, we may restrict ourselves to the case 
$$\lambda\le \kappa \rho.$$

For $t \ge \ell$ we proceed differently.  By the Cauchy-Schwarz inequality we get
$$
|T_t|^2 \le \#\cV_t \sum_{\nu\in \cV_t}\left| \sum_{\omega \in \cW_t}  \chi\(f_1\(\nu+\omega\)\)\psi\(f_2 \(\nu+\omega\)\)\right|^2.
$$

Extending the summation over $\nu$ to the linear subspace $\cL_k$ defined by~\eqref{Lk} and squaring out we derive
\begin{equation}
\label{eq:Cauchy} 
\begin{split}
|T_t|^2 \le \#\cV_t \sum_{\omega_1,\omega_2 \in \cW_t}  \sum_{\nu \in \cL_k}
& \chi\(f_1\(\nu+\omega_1\)/f_1\(\nu+\omega_2\)\) \\
 & \qquad  \psi\(f_2\(\nu+\omega_1\)-f_2\(\nu+\omega_2\)\).
\end{split}
\end{equation}

The rational function $f_1(X+\omega)/f_1(X)$ can be an $e$th power for some~$\omega\not=0$ and for these 'bad' $\omega$ we have to apply the trivial bound for all pairs $(\omega_1,\omega_2)$ with $\omega_2=\omega+\omega_1$. 
Since $f_1(X)$ has a simple root~$\xi$ this is only possible if $f_1(\xi+\omega)=f_1(\xi)=0$. However, this is only possible for at most $d_1=O(1)$ different values of $\omega \in \ovFq$. 
(Note that we use the convention that a rational function $f(X)=\frac{a(X)}{b(X)}$ is represented by polynomials $a(X)$ and~$b(X)$ with $\gcd(a(X),b(X))=1$.)

Recall that by our assumption $f_2(X) \in\cF_{d_2}(2)$,  for all $\omega_1\ne \omega_2$ the corresponding polynomial
$f_2(X+\omega_1)-f_2(X+\omega_2)$, 
is of the form allowed in Lemma~\ref{lem:weilsub}. 

To estimate the inner sum over $\nu$ in~\eqref{eq:Cauchy},  we now use the trivial bound~$2^k$ for $O\( \# \cW_t\)$ exceptional pairs 
$\(\omega_1, \omega_2\)$ as in the above (for example for $\omega_1=\omega_2$)
and use  Lemma~\ref{lem:weilsub} for the remaining $O\( \(\# \cW_t\)^2\)$  pairs 
$\(\omega_1, \omega_2\)$. Hence we obtain
$$
|T_t|^2 \ll \#\cV_t \( \(\# \cW_t\)^2 2^{r/2} +  \# \cW_t 2^{k} \).
$$
(We recall that the implied constants may depend on the degrees $d_1$ and~$d_2$).

Finally, using $\#\cV_t \# \cW_t \leq \#\Grs $, we obtain
$$
|T_t|^2  \ll \#\Grs  \#\cW_t 2^{r/2}  + \# \Grs 2^k.
$$
Since by~\eqref{eq: Card Vt Wt} and~\eqref{eq: AnyBinCoeff} 
$$
\# \cW_t  \le 2^{(1-\kappa)H^*\((\rho-\lambda)/(1-\kappa)\) r+o(r)}, 
$$
we now derive 
$$
|T_t| \ll \(2^{ \mathsf g(\rho; \lambda,\kappa)r}+2^{ \mathsf h(\rho; \lambda,\kappa)r}\)2^{o(r)}.
$$
Thus the total  contribution to $S_{s,r}(f)$ 
from such values of $t\ge \ell$ is bounded by 
\begin{equation}
\label{eq: Large t} 
 \sum_{t=\ell }^{s}|T_t| 
 \le s \(2^{\mathsf g(\rho; \lambda,\kappa)r}+2^{\mathsf h(\rho; \lambda,\kappa)r}\)2^{o(r)}\le \(2^{\mathsf g(\rho; \lambda,\kappa)r}+2^{\mathsf h(\rho; \lambda,\kappa)r}\) 2^{o(r)}.
\end{equation}
Combining~\eqref{eq: Small t}  and~\eqref{eq: Large t} we conclude the result.

\section{Characterisation of rational functions in $\cF_d(q)$} 

\subsection{Proof of Theorem~\ref{thm:mon}}\label{sec:mon}

We now show that  if  $\gcd(d,p)=1$ then 
the monomials $f(X)=X^d$ 
satisfy $f(X) \not\in \cF_d(q)$  if and only if 
\begin{equation}\label{eq:d} d \in \{1,2\}\cup \{p^k+1 :~k=1,2,\ldots\}.
\end{equation}
We may assume $\alpha\not=0$. 

For $d\in \{1,2\}$ the polynomial $(X+1)^d-X^d$ is trivially of the form $\beta X+1$.
So it remains to note that we can write  $1 = \mu^p - \mu$ for some~$\mu \in \ovFq$

So we now assume that $d \ge 3$.

 For $d=p^k+1$ with $k\ge 1$
we have
$$(X+1)^{p^k+1}-X^{p^k+1}=\sum_{j=0}^{p^k}\binom{p^k+1}{j}X^j.$$
By  Lemma~\ref{lem:lucas} we have 
$$\binom{p^k+1}{j}\equiv 0\bmod p \quad \mbox{for }j=2,3,\ldots,p^k-1$$
and 
$$\binom{p^k+1}{j}\equiv 1\bmod p \quad \mbox{for }j\in\{0,1,p^k\}.$$
Hence,
$$
(X+1)^{p^k+1}-X^{p^k+1}=X^{p^k}+X+1 =  g(X)^p - g(X) + 2X,
$$
where 
$$
g(X) = X^{p^{k-1}}+X^{p^{k-2}}+\ldots+X +\mu
$$
and as before $\mu$ is such that $1 = \mu^p-\mu$.

Next,  we check the remaining values of $d$ which are not of the form~\eqref{eq:d}. 
For $\omega \ne 0 $ the polynomial
$$f_{\omega}(X)=(X+\omega)^d- X^d=\omega d X^{d-1}+\ldots+ \omega^d
$$
has the leading coefficient $\omega d$ and hence 
is of degree $d-1\ge 2$ since $\gcd(d,p)=1$. 
Assume that 
$$f_{\omega}(X)=\alpha(g(X)^p-g(X))+\beta X.$$
Then we have   $d\equiv 1\bmod p$ and so the $p$-adic expansion of $d$ is
of the form
$$
d=1+d_1p+\ldots+d_kp^k,\quad \ 0\le d_1,\ldots,d_k<p, \quad k\ge 1, \quad d_k\ne0.
$$
By our assumption,  
we have either
\begin{equation}\label{dk1} d_k\ge 2, \quad \mbox{(that is, $p>2$)},
\end{equation}
and thus
\begin{equation}\label{dk2} d\ge 2p^k+1,
\end{equation}
or 
\begin{equation}\label{dk3}k\ge 2, \quad d_k=1,\quad \mbox{and}\quad d_i\ne 0 \quad
\mbox{ for some $i$ with $1\le i<k$},
\end{equation} 
and thus
\begin{equation}\label{eq: large d}
d\ge p^k + p^i + 1.
\end{equation}

Computing the derivative $f_{\omega}'(X)$ we  see that 
$$
f_{\omega}'(X) = -\alpha g'(X)+\beta
$$
and also 
\begin{align*}
f_{\omega}'(X)&=d((X+\omega)^{d-1}-X^{d-1})\\
&=d \sum_{j=1}^{d-1}\binom{d-1}{j}\omega^j X^{d-1-j}\\
&=d \sum_{j=2}^{d-1}\binom{d-1}{j}\omega^j X^{d-1-j},
\end{align*}
since for $j=1$ we have
$$
\binom{d-1}{1} = d-1\equiv 0\bmod p.
$$

Thus either $f_{\omega}'(X)=\beta$ or 
$$
\deg f_{\omega}' = \deg g'< p^{-1} \deg f_{\omega}=\frac{d-1}{p},
$$ 
that is, in either case,
\begin{equation}\label{eq:zero}
\binom{d-1}{j}\omega^j =0,\qquad  j=2,\ldots,d-1-(d-1)/p.
\end{equation}
In particular, 
in the case~\eqref{dk1} the inequality~\eqref{dk2} implies
$$d-1-\frac{d-1}{p}=(d-1)\(1-\frac{1}{p}\)\ge 2p^k\(1-\frac{1}{p}\)\ge p^k,$$
and in the case~\eqref{dk3}
the inequality~\eqref{eq: large d} implies 
$$
d-1-(d-1)/p = (d-1)(1-1/p)\ge (p^k+p)\left(1-\frac{1}{p}\right)\ge p^i.
$$
Now, applying~\eqref{eq:zero} with $j=p^k$ in the first case and $j = p^i$ in the second case, and invoking Lemma~\ref{lem:lucas} again, 
we derive
$$\binom{d-1}{p^k}\equiv \binom{d_k}{1}\not\equiv 0\bmod p$$
in the firs case, and
$$\binom{d-1}{p^i}\equiv\binom{d_i}{1}\equiv d_i \not\equiv 0\bmod p,$$
in the second case.
This implies $\omega^{p^k}=0$ and
$\omega^{p^i}=0$, respectively, and thus $\omega=0$ in both cases, 
which contradicts our assumption~$\omega\ne 0$. 

\subsection{Proof of Theorem~\ref{thm:inv}}\label{sec:inv}

First, for $\alpha=0$ and any $\omega,\beta\in \F_{q^r}$ with $\omega \ne 0$ we have 
to verify that
$$
(X+\omega)^{-d}-X^{-d}=\frac{X^d-(X+\omega)^d}{X^d(X+\omega)^d}\ne \beta X.
$$
This is trivial since $0$ is a pole of the left hand side but not of the right hand side.

It remains
to verify that for any $\beta\in \F_{q^r}$ and $\alpha,\omega\in \F_{q^r}^*$
$$\alpha^{-1}((X+\omega)^{-d}-X^{-d}-\beta X)=\frac{u(X)}{v(X)}$$
is not of the form $g(X)^p-g(X)$ with a rational function $g(X)\in \overline{\F_q}(X)$, where 
$$u(X)=-\alpha^{-1}\(\beta X^{d+1}(X+\omega)^d-X^d+(X+\omega)^d\)$$
and 
$$v(X)= X^d(X+\omega)^d.$$
Note that $u(0)u(-\omega)\not=0$ and thus $\gcd(u(X),v(X))=1$.
Write 
$$g(X)=\frac{a(X)}{b(X)}$$ 
with polynomials $a(X),b(X)\in \overline{\F_q}[X]$ such that  $\gcd(a(X),b(X))=1$ and $b(X)$ is monic. 
Cleaning the denominators, we get
$$b(X)^pu(X)=a(X)(a(X)^{p-1}-b(X)^{p-1})v(X).$$
Since $\gcd(b(X),a(X)^{p-1}-b(X)^{p-1})=1$ we get by the unique factorization theorem
$$v(X)=b(X)^p$$
and thus $d\equiv 0\bmod p$.

\subsection{Proof of Theorem~\ref{thm:pol}}\label{sec:pol}

We have to verify that under our conditions on $f(X)$ for any $\omega \ne 0$, 
the polynomial $f_\omega(X)$
is not of the form  $\alpha(g(X)^p-g(X))+\beta X$.

Suppose the contrary
$$f(X+\omega)-f(X)=\alpha(g(X)^p-g(X))+\beta X$$
and write
$$f(X)=\sum_{j=0}^d\gamma_jX^j\in \F_q[X],\quad \gamma_d\not=0.$$
Put
$$f_{\omega}(X)=f(X+\omega)-f(X).$$
We have either
\begin{equation}\label{eq: deg 1}
f_{\omega}(X)=\beta X
\end{equation}
or
\begin{equation}\label{eq: p|deg}
\deg f_{\omega} \equiv 0\bmod p.
\end{equation}
Hence, the coefficients $R_\ell$ of $f_{\omega}(X)$ at $X^\ell$ vanish for $\ell=\ell_0,\ldots,d-1$ where $\ell_0=2$ in the case of~\eqref{eq: deg 1}
 and $\ell_0=d-d_0+1$
in the case of~\eqref{eq: p|deg}. 

We have
$$R_\ell=\sum_{j=\ell+1}^d\gamma_j \binom{j}{\ell}\omega^{j-\ell} =0,\quad \ell=\ell_0,\ldots,d-1.$$
Note that by Lemma~\ref{lem:lucas} 
$$ 
\binom{d}{k}\equiv  \binom{d_0}{k}\not\equiv 0\bmod p,\qquad k=0,\ldots,d_0.
$$  
Define $T_\ell$, $\ell=1,\ldots,d_0-1$, recursively by setting $T_1=R_{d-1}$ and then for $\ell=2,\ldots,\min\{d-\ell_0,d_0+1\}$ by 
$$T_\ell=R_{d-\ell}-\gamma_d^{-1}\sum_{k=1}^{\ell-1}\gamma_{d-\ell+k}
 \binom{k+d-\ell}{ d-\ell}  \binom{d}{k}^{-1}T_k.$$
By induction we can show that
$$T_\ell=\gamma_d  \binom{d}{ \ell}\omega^\ell =0,\quad \ell=1,\ldots,\min\{d-\ell_0,d_0+1\}.$$
Since 
$$\gamma_d\ne 0 \mand  \binom{d}{ \ell}\not\equiv 0\bmod p, \quad \ell=1,\ldots,d_0,
$$ 
we obtain
$$\omega^\ell=0,\quad \ell=1,\ldots,\min\{d-\ell_0,d_0\}.$$
If $\min\{d-\ell_0,d_0\}\ge 1$, this implies $\omega= 0$, 
which contradicts our assumption  $\omega\ne 0$.
Considering both cases, $\ell_0=2$ and $\ell_0=d-d_0+1$, we get the desired contradiction if $d\ge 3$ and $d_0\ge 2$.

\subsection{Proof of Theorem~\ref{thm:ratsuff}}\label{sec:suff}
Assume that 
$$f(X+\omega)-f(X)-\beta X=\alpha\(\frac{a(X)^p}{b(X)^p}-\frac{a(X)}{b(X)}\)$$
for some $\alpha,\beta,\omega \in \overline{\F_q}$, $\omega\ne 0$, and relatively prime polynomials $a(X),b(X)\in \overline{\F_q}[X]$ with  $b(X)$ monic.
Then we get after simple calculations
\begin{equation}\label{uv}
\begin{split} 
\(u(X+\omega)v(X)-u(X)v(X+\omega)-\beta X v(X+\omega)v(X)\)&b(X)^p\\ 
=\alpha\(a(X)^p-a(X)b(X)^{p-1}\)&v(X)v(X+\omega).
\end{split} 
\end{equation}
Since $\gcd\(b(X),a(X)^p-a(X)b(X)^{p-1}\)=1$ and by our nondivisibility assumption of 
$ v(X)v(X+\omega)$  by  a nontrivial $p$th power, we obtain $b(X)=1$. 

For $\alpha=\beta=0$, since $\gcd(u(X),v(X))=1$, we get $v(X)=v(X+\omega)$. Comparing the coefficients of $X^{\deg{v}-1}$ we get $\deg v\equiv 0\bmod p$, which is excluded.

For $\alpha=0$ and $\beta\ne 0$, since $\deg u\le \deg v+1$, the left hand side 
of~\eqref{uv} is nonzero, which contradicts the right hand side.

For $\alpha\beta\ne 0$, the degree of the left hand side of~\eqref{uv} is $2\deg v+1$ but the 
degree of the right hand side is $\equiv 2\deg v\bmod p$, which is again not possible.

\section*{Acknowledgments} 

The authors wish to thank Alina Ostafe for useful discussions.

L.~M.\ was supported by the Austrian Science Fund FWF under the Grants 
P~31762 and F~5506, which is part of the Special Research Program 
"Quasi-Monte Carlo Methods: Theory and Applications". I.~E.~S. was supported by the Austrian Research Council ARC Grant DP~200100355.

 \end{document}